\documentclass[a4paper,11pt]{amsart}
\usepackage{pgf}
\usepackage{tikz}
\usetikzlibrary{arrows,automata}
\usepackage[utf8]{inputenc}
\usepackage{amsmath,amsfonts,amsthm,amssymb,mathtools}
\usepackage{xfrac}
\usepackage{verbatim}
\usepackage[english]{babel}
\usepackage{todonotes}
\usepackage[backref]{hyperref}
\usepackage{color}
\definecolor{mylinkcolor}{rgb}{0.5,0.0,0.0}
\definecolor{myurlcolor}{rgb}{0.0,0.0,0.75}
\hypersetup{colorlinks=true,urlcolor=myurlcolor,citecolor=myurlcolor,linkcolor=mylinkcolor,linktoc=page,breaklinks=true}
\usepackage{cite}
\usepackage{cleveref} 
\usepackage{fullpage}
\usepackage{multicol}

\theoremstyle{plain}
  \newtheorem{theorem}{Theorem}[section]
  \newtheorem{lemma}[theorem]{Lemma}
  
  \newtheorem{proposition}[theorem]{Proposition}
  
  \newtheorem{corollary}[theorem]{Corollary}
  
\theoremstyle{definition}

\theoremstyle{remark}
  \newtheorem{remark}[theorem]{Remark}

\newcommand{\field}[1]{\mathbb{#1}}  
\newcommand{\Q}{\field{Q}} 
\newcommand{\C}{\field{C}} 
\newcommand{\Z}{\field{Z}} 
\newcommand{\F}{\field{F}} 
\newcommand{\Qbar}{\overline{\Q}}

\renewcommand{\P}{\field{P}}

\renewcommand{\H}{\field{H}} 

\newcommand{\Jac}{\operatorname{Jac}}
\newcommand{\rk}{\operatorname{rk}}

\DeclareMathOperator{\Div}{div}

\DeclareMathOperator{\PSL}{PSL}
\DeclareMathOperator{\SL}{SL}

\DeclareMathOperator{\Pic}{Pic}

\DeclareMathOperator{\Gal}{Gal}

\newcommand{\eps}{\varepsilon} 
\newcommand{\p}{\mathfrak p}

\title{Torsion subgroups of elliptic curves over\\quintic and sextic number fields}
\author{Maarten Derickx and Andrew V. Sutherland}

\thanks{The second author was supported by NSF grant DMS-1522526.}

\begin{document}

\begin{abstract}
Let $\Phi^\infty(d)$ denote the set of finite abelian groups that occur infinitely often as the torsion subgroup of an elliptic curve over a number field of degree~$d$.  The sets $\Phi^\infty(d)$ are known for $d\le 4$.  In this article we determine $\Phi^\infty(5)$ and $\Phi^\infty(6)$.
\end{abstract}

\maketitle

\section{Introduction}
Let $E$ be an elliptic curve over a number field $K$.  By the Mordell-Weil theorem \cite{Weil29}, the set of $K$-rational points on $E$ forms a finitely generated abelian group $E(K)$.
In particular, its torsion subgroup $E(K)_{\rm tors}$ is finite, and it is well known that it can be generated by two elements \cite{Silverman09}.
There thus exist integers $m,n\ge 1$ such that
\[
E(K)_{\rm tors}\simeq \Z/m\Z \oplus \Z/mn\Z.
\]
The uniform boundedness conjecture states that for every number field $K$ there exists a bound~$B$ such that $\#E(K)_{\rm tors}\le B$ for every elliptic curve $E$ over $K$.
This conjecture is now a theorem due to Merel \cite{Merel96}, who actually proved the strong version of this conjecture in which the bound~$B$ depends only on the degree $d\coloneqq[K\!:\Q]$.
It follows that for every positive integer $d$, there~is a finite set $\Phi(d)$ of isomorphism classes of torsion subgroups that arise for elliptic curves over  number fields of degree $d$; we may identify elements of $\Phi(d)$ by pairs of positive integers $(m,mn)$.

The set $\Phi(1)$ was famously determined by Mazur \cite{Mazur77a}, who proved that
\[
\Phi(1)=\{(1,n):1\le n\le 12,\ n\ne 11\}\ \cup\ \{(2,2n):1\le n\le 4\}.
\]
The set $\Phi(2)$ was determined in a series of papers by Kenku, Momose, and Kamienny, culminating in \cite{KM88,Kamienny92}, that yield the result
\[
\Phi(2)=\{(1,n):1\le n\le 18,\ n\ne 17\}\ \cup\  \{(2,2n):1\le n\le 6\}\ \cup\ \{(3,3),(3,6),(4,4)\}.
\]
For $d>2$ the sets $\Phi(d)$ have yet to be completely determined.
However, if we distinguish the subset $\Phi^\infty(d)\subseteq \Phi(d)$ of torsion subgroups that arise for infinitely many $\Qbar$-isomorphism classes of elliptic curves defined over number fields of degree $d$, we can say more.

We have $\Phi^\infty(1)=\Phi(1)$ and $\Phi^\infty(2)=\Phi(2)$.  In \cite{JKS04} Jeon, Kim, and Schweizer found
\[
\Phi^\infty(3) =\, \{(1,n):1\le n\le 20, \ n\ne 17,19\}\ \cup\  \{(2,2n):1\le n\le 7\},
\]
and in \cite{JKP06} Jeon, Kim and Park obtained
\begin{align*}
\Phi^\infty(4) =\ &\{(1,n):1\le n\le 24, \ n\ne 19,23\}\ \cup\ \{(2,2n):1\le n\le 9\}\\
&\cup\  \{(3,3n):1\le n\le 3\}\ \cup \{(4,4),(4,8),(5,5),(6,6)\}.
\end{align*}
In this article we determine the sets $\Phi^\infty(5)$ and $\Phi^\infty(6)$.

\begin{theorem}\label{thm:main}
Let $\Phi^\infty(d)$ denote the set of pairs $(m,mn)$ for which $E(K)_{\rm{tors}}\simeq \Z/m\Z\oplus\Z/mn\Z$
for infinitely many non-isomorphic elliptic curves $E$ over number fields $K$ of degree $d$.  Then
\[
\Phi^\infty(5)=\,\{(1,n):1\le n\le 25, \ n\ne 23\}\ \cup\ \{(2,2n):1\le n \le 8\},
\]
and
\begin{align*}
\Phi^\infty(6) =\ &\{(1,n):1\le n\le 30,\ n\ne 23,25,29\}\ \cup\ \{(2,2n):1\le n\le 10\}\\
&\cup\ \{(3,3n):1\le n\le 4\}\ \cup \{(4,4),(4,8),(6,6)\}.
\end{align*}
\end{theorem}

For $d=5,6,7,8$, the elements $(1,n)\in \Phi^\infty(d)$ were determined in \cite{DvH13} using a strategy that we generalize here.
The key steps involve computing (or at least bounding) the gonalities of certain modular curves, and determining whether their Jacobians have rank zero or not.
To obtain gonality bounds we require explicit models for the modular curves $X_1(m,mn)$ that parametrize triples $(E,P,Q)$, where $E$ is an elliptic curve with independent points $P$ of order $m$ and $Q$ of order $mn$; for our approach to be computationally feasible, it is important that these models have low degree and reasonably small coefficients.
Optimized models for $X_1(n)=X(1,n)$ for $n\le 50$ were constructed in \cite{Sutherland12}.
Here we extend the approach of \cite{Sutherland12} to construct optimized models for $X_1(m,mn)$ for $m^2n\le 120$, as described in \S\ref{sec:models}; these can be found at \cite{models}.
These models are necessarily defined only over the cyclotomic field $\Q(\zeta_m)$.
The need to work over $\Q(\zeta_m)$ requires us to develop some new techniques for determining when the Jacobian of $X_1(m,mn)$ has rank zero over $\Q(\zeta_m)$; these are described in \S\ref{sec:Jrank}.
The case $X_1(2,30)$ proved to be computationally challenging, so we used an alternative strategy based on ideas in \cite{DKM16}, as explained in \S\ref{sec:proofs}.

In principle our methods can also determine $\Phi^\infty(7)$; we have computed explicit models for all of the relevant modular curves and proved that their Jacobians have rank zero.
We can prove $(2,2n)\in \Phi^\infty(7)$ for $1\le n\le 10$ and not for $n>15$; it remains only to determine for $11\le n\le 15$ whether the gonality of $X_1(2,2n)$ is greater than~$7$ or not (we expect the answer is yes in each case).
Similar comments apply to $\Phi^{\infty}(8)$ but not $\Phi^{\infty}(9)$, which requires new techniques, as explained in \Cref{remark:higher_d}.

The source code for computations on which results of this article depend is available at \cite{mdmagma}.

\section{Background}\label{sec:background}

In this section we briefly recall background material and introduce some notation. Let $K$ be a field and let $X/K$ be a \emph{nice curve}, by which we mean that $X/K$ is of dimension 1, smooth, projective, and geometrically integral.
The \emph{gonality} $\gamma(X)$ of $X$ is the minimal degree of a finite $K$-morphism $X\to \P^1_K$.  If $L/K$ is a field extension and $X_L:=X\times_K L$ is the base change of $X$ to $L$, we necessarily have $\gamma(X_L)\le \gamma(X)$, and we call $\gamma(X_L)$ the $L$-\emph{gonality} of $X$. If $K$ is a number field, $\p$ is a prime of $K$ of good reduction for $X$, and $X_{\F_\p}$ is the reduction of $X$ to the residue field $\F_\p$ of $\p$, then $\gamma(X_{\F_\p})\le \gamma(X)$, and we call $\gamma(X_{\F_\p})$ the $\F_\p$-\emph{gonality} of $X$.

\begin{proposition}[Abramovich, Kim-Sarnak]\label{thm:Abramovich}
Let $\Gamma\subseteq \PSL_2(\Z)$ be a congruence subgroup.  The $\C$-gonality of the modular curve $X_\Gamma$ is at least $(\lambda_1/24)[\PSL_2(\Z):\Gamma]$, where $\lambda_1\ge 975/4096$.
\end{proposition}
\begin{proof}
See \cite[Thm.\,0.1]{Abramovich96} for the first statement and \cite[p.\,176]{KS03} for the lower bound on $\lambda_1$.
\end{proof}

Let $d(X)$ denote the least integer for which the set $\{a\in X(\overline{K}):[K(a):K]=d\}$ of \emph{points of degree $d$} on $X$ is infinite. We have the following result of Frey \cite{Frey94}, which can be viewed as a corollary of Faltings' proof of Lang's conjecture \cite{Faltings94}.

\begin{proposition}[Frey]\label{thm:Frey}
Let $X$ be nice curve over a number field.  Then $d(X)\le \gamma(X)\le 2d(X)$.
\end{proposition}
\begin{proof}
If $f\in K(X)$ is a function of degree $d$, then by Hilbert irreducibility there are infinitely many points of degree $d$ over $K$ among the roots of $f-c$ as $c$ varies over $K$; this proves the first inequality, and the second is \cite[Prop.\,1]{Frey94}.
\end{proof}

There is one situation in which the lower bound of \Cref{thm:Frey} is known to be tight.

\begin{proposition}\label{lem:rank0}
Let $X/K$ be a nice curve whose Jacobian has rank $0$.  Then $d(X)=\gamma(X)$.
\end{proposition}
\begin{proof}
Let $d < \gamma(X)$ be a positive integer.  The map $\pi\colon X^{(d)}\to\Jac(X)$ from the $d$th symmetric power of $X$ to its Jacobian is injective, since otherwise we could construct a function $f\in K(X)$ of degree $d<\gamma(X)$ from the difference of two linearly equivalent divisors of degree $d$ and view $f$ as a map $X\to\P^1_K$.  If $\rk( \Jac(X)(K))=0$ then $\pi(X^{(d)}(K))$ is finite, and so is $X^{(d)}(K)$; it follows that $X$ has only finitely many points of degree $d$.
\end{proof}

The proofs of \Cref{thm:Frey,lem:rank0} together imply the following corollary.

\begin{corollary}\label{cor:rank0}
Let $X/K$ be a nice curve over a number field.  If $K(X)$ contains a function of degree~$d$ then $X$ has infinitely many points of degree~$d$.  When $\rk(\Jac(X)(K))=0$ the converse also holds.
\end{corollary}

For positive integers $m$ and $n$ we use $Y_1(m,mn)$ to denote the modular curve that parameterizes triples $(E,P,Q)$, where $E$ is an elliptic curve with independent points $P$ of order $m$ and~$Q$ of order $mn$, and $X_1(m,mn)$ is its projectivization obtained by adding cusps.  We view $X_1(m,mn)$ as a $\Z[\frac{1}{mn}]$-scheme that is isomorphic to the coarse moduli space for the corresponding algebraic stack $\mathcal{X}_1(m,mn)$ in the sense of \cite{DR73}; all the cases of interest to us have $mn \ge 5$ and are also fine moduli spaces.
By fixing a primitive $m$th root of unity $\zeta_m$ in the function field of $X_1(m,mn)$, we may also view $X_1(m,mn)$ as a nice curve over $\Q(\zeta_m)$ that has good reduction at all primes not dividing $mn$, and we write $J_1(m,mn)$ for the Jacobian of $X_1(m,mn)$ as a curve over $\Q(\zeta_m)$.  There is an associated congruence subgroup
\[
\Gamma_1(m,mn):=\left\{\left(\begin{smallmatrix}a&b\\c&d\end{smallmatrix}\right)\in \SL_2(\Z):a\equiv d\equiv 1\bmod mn,\ c\equiv 0\bmod mn,\ b\equiv 0\bmod m\right\},
\]
and after fixing an embedding $\Q(\zeta_m)\hookrightarrow \C$, the quotient $\H^*/\Gamma_1(m,mn)$ of the extended upper half-plane by the action of $\Gamma_1(m,mn)$ is a compact Riemann surface isomorphic to $X_1(m,mn)(\C)$.
The image of $\Gamma_1(m,mn)$ in $\PSL_2(\Z)$ has index $\frac{m^3n^2}{2}\prod_{p|mn}(1-\frac{1}{p^2})\ge m^2n-1$.

Let $X_{0,1}(m,mn)$ be the projectivisation of the modular curve that parametrizes triples $(E,G,Q)$, where $E$ is an elliptic curve, $G$ is a cyclic subgroup of order $m$ (or equivalently, a cyclic isogeny of degree $m$) and~$Q$ is an independent point of order $mn$ (so $G\cap \langle Q\rangle=\{0\}$).
The curve $X_{0,1}(m,mn)_{\Q(\zeta_m)}$ is isomorphic to $X_{1,1}(m,mn)$ (as can be seen by considering the corresponding congruence subgroups, or by writing down a natural transformation between the two functors on schemes over $\Z/nm\Z$, where we view $X_{0,1}(m,mn)_{\Q(\zeta_m)}$ as parametrizing quadruples $(E,G,Q,\zeta_m)$ with $E,G,Q$ as above and $\zeta_m$ a chosen primitive $m$th root of unity).  Unlike $X_1(m,mn)$, the curve $X_{0,1}(m,mn)$ has the advantage of always being defined over $\Q$.
Now let $X_1(m^2n):=X_1(1,m^2n)$ parameterize pairs $(E,Q)$ in which $E$ is an elliptic curve with a point~$Q$ of order $m^2n$, and consider the map
\[
\varphi\colon X_1(m^2n)\to X_{0,1}(m,mn)
\]
that sends the pair $(E,Q)$ to the triple $(E/\langle mnQ\rangle,\,E[m]/\langle mnQ\rangle,\, Q \bmod \langle mnQ\rangle)$.
The group $E[m]/\langle mnQ\rangle$ and the point $Q \bmod \langle mnQ\rangle $ are independent because $Q \bmod \langle mnQ\rangle $ has the same order $mn$ as $Q \bmod E[m]$.
The map $\varphi$ is defined over $\Q$ and has degree $m$.  The group $(\Z/m^2n\Z)^\times$ acts on $X_1(m^2n)$ via the diamond operators $\langle a\rangle\colon (E,P)\mapsto (E,aP)$.  We have $a\equiv 1\bmod mn$ precisely when $\langle a\rangle$  stabilizes $\varphi$, meaning $\varphi=\varphi\circ\langle a\rangle$, and the quotient of $X_1(m^2n)$ by this automorphism subgroup is isomorphic to $X_{0,1}(m,mn)$.

Throughout this article $\Qbar$ denotes a fixed algebraic closure of $\Q$ that contains all number fields $K$ under consideration, and we identify $\overline{K} = \Qbar$.
For a modular curve $X$ defined over a number field $K$, we define the \emph{degree over $\Q$} of a point $a\in X(\overline{\Q})$ to be the absolute degree $[L:\Q]$ of the minimal extension $L/K$ for which $a\in X(L)$ and let $\Q(a)$ denote the field $L$ (which contains $K$).  For the sake of clarity we may refer to $[L:K]$ as the \emph{degree of $a$ over $K$}.

Finally, we use $\phi(m)\coloneqq[\Q(\zeta_m):\Q]=\#(\Z/m\Z)^\times$ throughout to denote the Euler function.

\section{Constructing models of \texorpdfstring{$X_1(m,mn)$}{X1(m,mn)}}\label{sec:models}
Our method for constructing explicit methods of $X_1(m,mn)$ is a generalization of the technique used in \cite{Sutherland12} to construct models for $X_1(n)\coloneqq X_1(1,n)$, which we now briefly recall.
Given $n>3$ one begins as in \cite{Reichert86} with the universal family of elliptic curves
\[
E(b,c):=y^2+(1-c)y-by=x^3-bx^2
\]
in Tate normal form with rational point $P=(0,0)$ and imposes the constraint
\[
\left\lceil\frac{n+1}{2}\right\rceil P + \left\lfloor\frac{n-1}{2}\right\rfloor P = 0
\]
by requiring the $x$-coordinates of the two summands to coincide (the $y$-coordinates of the summands cannot coincide because $\lceil (n+1)/2\rceil\ne \lfloor(n-1)/2\rfloor$, so the points must sum to zero).
After clearing denominators and removing spurious factors corresponding to torsion points whose order properly divides $n$, one obtains a (singular) affine plane curve $C/\Q$ with the same function field as $X_1(n)$.
Each non-singular point $(b_0,c_0)$ on this curve determines an elliptic curve $E(b_0,c_0)$ on which $P=(0,0)$ is a point of order $n$.
The equations obtained by this method are typically much larger than necessary, but the algorithm in \cite{Sutherland12} can be used to obtain models with lower degrees, fewer terms, and smaller coefficients.

\subsection{Constructing models using elliptic surfaces}\label{sec:specialmodels}
We use a similar approach to construct equations for $X_1(m,mn)$.
For $m=2$ we use the parameterized family of elliptic curves constructed by Jain in \cite{Jain10}, in which the elliptic curve
\[
E_2(q,t):\quad y^2 = x^3+(t^2-qt-2)x^2-(t^2-1)(qt+1)^2x,
\]
has the rational point $P_2:=(0,0)$ of order 2 and the rational point
\[
Q_2(q,t) := \bigl((t+1)(qt+1)\,,\,t(qt+1)(t+1)\bigr)
\]
of infinite order; see \cite[Thm.\,1.ii]{Jain10}.
As suggested to us by Noam Elkies, for any $n>1$, setting
\[
\left\lceil\frac{2n+1}{2}\right\rceil Q_2(q,t) + \left\lfloor\frac{2n-1}{2}\right\rfloor Q_2(q,t)=0,
\]
allows us to construct a model for $X_1(2,2n)$; as above, it is enough to equate the $x$-coordinates of the two summands, and this yields a polynomial equation in $q$ and $t$.
With $n=7$, for example, after clearing denominators and removing spurious factors we obtain the equation
\begin{align*}
7q^{12}t^4 &+ 56q^{11}t^3 + 70q^{10}t^4 + 112q^{10}t^2 + 208q^9t^3 + 64q^9t - 111q^8t^4 + 144q^8t^2\\
&- 624q^7t^3 - 156q^6t^4 - 1104q^6t^2 - 512q^5t^3 - 832q^5t - 55q^4t^4 - 592q^4t^2\\
&- 256q^4 - 136q^3t^3 - 256q^3t - 10q^2t^4 - 96q^2t^2 - 16qt^3 - t^4 = 0.
\end{align*}
This equation is not as compact as we might wish, and its degree in both~$q$ and $t$ is greater than the gonality of $X_1(2,14)$, which is $3$.
However, after applying the optimizations described in \cite{Sutherland12} we obtain the equation
\[
(u^2 + u)v^3 + (u^3 + 2u^2 - u - 1)v^2 + (u^3 - u^2 - 4u - 1)v - u^2 - u = 0,
\]
whose degree in $u$ and $v$ matches the gonality of $X_1(2,14)$.
The relation between the $(u,v)$ coordinates and the $(q,t)$ coordinates is given by
\[
q = \frac{u+v}{v-u},\qquad\qquad t = \frac{(u-v)(u+v)(u+v+2)}{u^3 + u^2v + 2u^2 + uv^2 + 2uv + v^3 + 2v^2}.
\]
The reader may wish to compare this model for $X_1(2,14)$ with the one given in \cite[p.\,589]{JKL11a}.

A similar approach can be used to obtain equations for $X_1(3,3n)$.
From \cite[Thm.\,4.ii]{Jain10} we have the parameterized family of elliptic curves
\[
E_3(q,t):\quad y^2 + (qt-q+t+2)xy+(qt^2-qt+t)y=x^3,
\]
with the rational point $P_3:=(0,0)$ of order 3 and the rational point
\[
Q_3(q,t):=(-t,t^2)
\]
of infinite order.
For $n\ge 1$, equating the $x$-coordinates of $\left\lceil\frac{3n+1}{2}\right\rceil Q_3(q,t)$ and $\left\lfloor\frac{3n-1}{2}\right\rfloor Q_3(q,t)$ yields an equation $F(q,t)=0$ that we may use to construct a model for $X_1(3,3n)$.
In order to obtain a geometrically integral curve we must factor $F$ over $\Q(\zeta_3)$ rather than $\Q$ (if we only factor over $\Q$, over $\Q(\zeta_3)$ we will have the union of two curves that are both birationally equivalent to $X_1(3,3n)$).

To obtain equations for $X_1(4,4n)$ we use a parameterization due to Kumar and Shioda; see the example following \cite[Rem.\ 10]{KS13}.
We have the family
\[
E_4(q,t):\quad y^2 + xy + (1/16)(q^2-1)(t^2-1)y = x^3 + (1/16)(q^2-1)(t^2-1)x^2,
\]
with the rational point $P_4:=(0,0)$ of order 4 and the rational point
\[
Q_4(q,t):= \bigl((q+1)(t^2-1)/8\,,\,(q+1)^2(t-1)^2(t+1)/32\bigr)
\]
of infinite order.  For any $n\ge 1$, equating the $x$-coordinates of $\left\lceil\frac{4n+1}{2}\right\rceil Q_4(q,t)$ and $\left\lfloor\frac{4n-1}{2}\right\rfloor Q_4(q,t)$ yields an equation $F(q,t)=0$ that we may use to construct a model for $X_1(4,4n)$ after factoring $F(q,t)$ over $\Q(\zeta_4)=\Q(i)$.

\begin{remark}\label{rem:verify}
It is usually obvious which factor of $F(q,t)$ is the correct choice (the biggest one), but one can verify the correct choice by checking that it yields a curve of the same genus as $X_1(m,mn)$.  The other non-conjugate factors of $F(q,t)$ correspond to modular curves $X$ that admit a non-constant map from $X_1(m,mn)$ of degree greater than $1$; provided $X_1(m,mn)$ has genus $g>1$, the Riemann-Hurwitz formula implies $g > g(X)$.  For $g\le 1$ one can instead prove that none of the non-conjugate factors yield a model for $X_1(m,mn)$ by finding a non-singular point that does not yield a triple $E(E,P,Q)$ with $P$ and $Q$ of the correct order (always possible).
\end{remark}

\subsection{A general method}\label{sec:generalmodels}
The methods in the previous section for $m=2,3,4$ rely on parameterizations obtained from elliptic surfaces that do not exist in general.
We now sketch a general method that works for any $m>3$.
Rather than constructing a model for the curve $X_1(m,n)$, we will construct a model for its quotient by the involution $(E,P,Q)\to(E,-P,Q)$, which we denote $X_1(m,n)^+$.
The curve $X_1(m,n)^+$ is defined over $\Q(\zeta_m)^+$, the maximal real subfield of $\Q(\zeta_m)$, and its base change from $\Q(\zeta_m)^+$ to $\Q(\zeta_m)$ is isomorphic to $X_1(m,n)$; for $m=1,2,3,4,6$ we have $\Q(\zeta_m)^+=\Q$ and
$X_1(m,n)^+\simeq X_{0,1}(m,n)$.
For the sake of brevity we give details only for the cases in which $X_1(m)$ has genus~0, which suffices for our purposes.

Let us fix $m>3$ and $n\ge 1$.
We may view $E(b,c)$ as the universal elliptic curve with rational points $P=(0,0)$ and $Q=(x,y)$; let $e(b,c,x,y)=0$ be the equation defining $E(b,c)$.
Let $f(b,c)=0$ be an equation for $X_1(m)$ constructed as in \cite{Sutherland12}, and let $h(b,c,x)$ denote the irreducible polynomial obtained by equating the $x$-coordinates of $\lceil(mn+1)/2\rceil Q$ and $\lfloor(mn-1)/2\rfloor Q$ and removing spurious factors as above.  The equations $e(b,c,x,y)=f(b,c)=h(b,c,x)=0$ define a curve in $\mathbb{A}^4[b,c,x,y]$; this curve is not reduced (because there are two choices for the unconstrained $y$-coordinate), but it contains the curve we seek.

Let us now assume $m\in \{4,5,6,7,8,9,10,12\}$ so that $g(X_1(m))=0$, in which case we may write $b=b(t)$ and $c=c(t)$ as functions of a single rational parameter $t$, as in \cite[Table\,3]{Kubert76}.
In this case there is no $f(b,c)$ to compute, our curve equation becomes $e(t,x,y)=0$, and we compute $h(t,x)$ by equating the $x$-coordinates of $\lceil(mn+1)/2\rceil Q$ and $\lfloor(mn-1)/2\rfloor Q$ and removing spurious factors as above.
Let $H$ be the resultant of $e$ and $h$ with respect to the variable $y$; the polynomial $H(t,x)$ is the square of a polynomial $F\in\Z[t,x]$ that is irreducible over $\Q$ but splits into $\phi(m)/2$ geometrically irreducible factors over $\Q(\zeta_m)^+$ that are $\Gal(\Q(\zeta_m)^+/\Q)$-conjugates.
We may take any of these factors as our model for $X_1(m,n)^+$.
The base change of this model from $\Q(\zeta_m)^+$ to $\Q(\zeta_m)$ is then a model for $X_1(m,mn)$.

By combining this method with \S\ref{sec:specialmodels} and applying the algorithm in \cite{Sutherland12}, we have constructed optimized models for $X_1(m,mn)$ for $m^2n\le 120$ and verified them as explained in \Cref{rem:verify}; these models are available in electronic form at \cite{models}.

\section{Modular Jacobians of rank zero over cyclotomic fields}\label{sec:Jrank}
In \cite{CES03}, Conrad, Edixhoven, and Stein give an explicit method for computing $L$-ratios for modular forms on $X_1(n)$.
By applying the proven parts of the Birch and Swinnerton-Dyer conjecture (see \cite[Cor.\,14.3]{Kato04} or \cite{KL89}) they are then able to prove that the rank of $J_1(p)\coloneqq\Jac(X_1(p))$ is zero over $\Q$ for all primes $p < 73$ except for $p=37, 43,53,61, 67$; see \cite[\S6.1.3,\ \S6.2.2]{CES03}. The software developed by Stein for performing these computations is available in the computer algebra system Magma \cite{magma} which can compute provably correct bounds on $L$-ratios as exact rational numbers;
in particular, one can unconditionally determine when the $L$-ratio is nonzero, in which case the rank is provably zero.

Here we adapt this method to prove that $J_1(m,mn)$ has rank $0$ over $\Q(\zeta_m)$ for suitable values of $m$ and $n$.
For the sake of brevity we focus on the cases $m=2,3,4,6$ in which $\Q(\zeta_m)$ has degree at most 2, which suffices for our application; the method generalizes to arbitrary $m$ and abelian extensions $K/\Q$ that contain an $m$th root of unity.
This includes $m=5$, which we list below for reference but do not need to prove our main result.

\begin{theorem}\label{thm:rank0}
The rank of $J_1(m,mn)$ is zero over $\Q(\zeta_m)$ if any of the following hold:
\begin{multicols}{2}
\begin{itemize}
\item $m=1$ and $n \leq 36$;
\item $m=2$ and $n \leq 21$;
\item $m=3$ and $n \leq 10$;
\item $m=4$ and $n \leq 6$;
\item $m=5$ and $n \leq 4$;
\item $m=6$ and $n \leq 5$.
\end{itemize}
\end{multicols}
\end{theorem}
All computations referred to in the proof below were performed using the \texttt{IsX1mnRankZero} function implemented in \cite{mdmagma}, which uses the \texttt{LRatio} function in Magma \cite{magma} to determine when the $L$-ratio is nonzero.

\begin{proof}
Let $J$ be the Jacobian of the quotient of the curve $X_1(m^2n)$ by the subgroup of diamond operators that stabilize the map $\varphi\colon X_1(m^2n)\to X_{0,1}(m,mn)$, as described in \S\ref{sec:background}.
We then have $J_1(m,mn) \simeq J_{\Q(\zeta_m)}$, and it suffices to prove that the rank of $J_{\Q(\zeta_m)}(\Q(\zeta_m)) = J(\Q(\zeta_m))$ is zero. 

For $m\le 2$ we have $\Q(\zeta_m)=\Q$ and the strategy of \cite{CES03} can be applied directly; the desired result follows from a computation that finds the $L$-ratios to be nonzero for all modular forms~$f$ corresponding to simple isogeny factors of $J$, for $m\le 2$ and $n$ as in the theorem.

We now assume $m\in \{3,4,6\}$ so that $\Q(\zeta_m)$ is a quadratic extension of $\Q$. Let $f$ be a newform of level dividing $m^2n$ such that the abelian variety  $A_f$ associated to $f$ is an isogeny factor of $J$. Proving that $\rk(J(\Q(\zeta_m)))=0$ is equivalent to proving that $\rk(A_f(\Q(\zeta_m)))=0$ for all such $f$.

Let $A$ be the Weil descent of $A_{f,\Q(\zeta_m)}$ down to $\Q$.
From the definition of the Weil descent we have $A(\Q) = A_f(\Q(\zeta_m))$, and the identity map $A_{f,\Q(\zeta_m)} \to A_{f,\Q(\zeta_m)}$ over $\Q(\zeta_m)$ induces a morphism $A_f \to A$ over $\Q$; it follows that $A$ is isogenous to $A_f \oplus A/A_f$. Let $\chi_m : (\Z/m\Z)^\times \to \Q^\times$ denote the quadratic character of $\Q(\zeta_m)$. One sees that $A/A_f$ is isogenous to $A_{f_{\chi_m}}$ by comparing traces of Frobenius on their Tate modules. In particular, $\rk(A_{f}(\Q(\zeta_m)))=0$ if and only  if both $\rk(A_f(\Q))=0$ and $\rk(A_{f_{\chi_m}}(\Q))=0$ hold.

The theorem now follows from a computation; we find that the $L$-ratios of $f$ and ${f_{\chi_m}}$ are nonzero for $m=3,4,6$ and $n$ as in the theorem and all newforms $f$ of level dividing $m^2n$ such that the abelian variety $A_f$ associated to $f$ is an isogeny factor of $J$. 

In fact (as kindly pointed out to us by the referee), it suffices to compute the $L$-ratios of the forms $f$; the twists $f_{\chi_m}$ already arise among the untwisted $f$ under consideration.  Indeed, it follows from \cite[Prop 3.1]{AtkinLi78} that for any Dirichlet character $\eps$ of conductor dividing $mn$, if $f$ is a modular form in $S_2(\Gamma_0(m^2n),\eps)$, then so is $f_{\chi_m}$, since $\chi_m^2=1$.
\end{proof}
\begin{remark}
We also computed the relevant $L$-ratios for $m\le 6$ and $n$ just past the range listed in \Cref{thm:rank0}; for each of these $(m,n)$ we found that at least one relevant $L$-ratio was zero. 
\end{remark}

\section{Proof of the main theorem}\label{sec:proofs}

The first step in our proof \Cref{thm:main} is to show for $d=5,6$ that $\Q(\zeta_m)(X_1(m,mn))$ contains a function of degree $d/\phi(m)$ if and only if $(m,mn)$ is one of the pairs for $\Phi^\infty(d)$ appearing in the theorem.
The following lemma illustrates why the forward implication is useful.

\begin{lemma}\label{lem:suff}
If $\Q(\zeta_m)(X_1(m,mn))$ contains a function of degree $\frac{d}{\phi(m)}$ then $(m,mn)\in \Phi^\infty(d)$.
\end{lemma}
\begin{proof}
We first note that for $mn \le 4$, either $\phi(m)=1$ and $(m,mn)\in \Phi(1)=\Phi^\infty(1)\subseteq\Phi^\infty(d)$, or $\phi(m)=2$ divides $d$ and $(m,mn)\in\Phi(2)=\Phi^\infty(2)\subseteq\Phi^\infty(d)$; in both cases the lemma holds.

We now assume $mn\ge 5$, in which case $X_1(m,mn)$ is a fine moduli space.  By Merel's proof of the uniform boundedness conjecture \cite{Merel96} there is a positive integer $B$ such that for all number fields $K$ of degree $d$ and elliptic curves $E/K$ we have $E(K)_{\rm tors} \subseteq E[B]$.  The integer $B$ is necessarily divisible by $mn$, since $X_1(m,mn)$ has points of degree $d$ over $\Q$.

Now let $f\in \Q(\zeta_m)(X_1(m,mn))$ be a function of degree $d/\phi(m)$.
For each $a \in \Q(\zeta_m)$ the points in $f^{-1}(a)$ have degree at most $d/\phi(m)$ over $\Q(\zeta_m)$, hence degree at most $d$ over $\Q$. By Hilbert irreducibility, there are infinitely many $a$ for which the points in $f^{-1}(a)$ have degree exactly $d$, but we also need to show that there are infinitely many $a$ for which the torsion subgroups of the elliptic curves corresponding to the points in $f^{-1}(a)$ are actually isomorphic to $\Z/m\Z \times \Z/mn\Z$ and not any larger.
In order to show this we consider the maps
\[
X_1(B,B) \overset{\pi}{\longrightarrow} X_1(m,mn)\overset{f}{\longrightarrow} \P^1,
\]
where $\pi$ sends $(E,P,Q)$ to $(E,(B/m)P, B/(mn)Q)$, and let $\varphi\coloneqq f\circ\pi$.
Let $A \subseteq \P^1(\Q(\zeta_m))$ be the set of $a$ for which every $b\in \varphi^{-1}(a)$ has degree $d\deg\pi$ over $\Q$.
The set $A$ is infinite, by Hilbert irreducibility, and for $a\in A$ every $c \in f^{-1}(a)$ has degree $d$ over $\Q$.
We claim that for all such $c=(E,P,Q)$ we have $E(\Q(c))_{\rm tors}=\langle P,Q\rangle$, which implies $(m,mn)\in \Phi^\infty(d)$ as desired.

Suppose not.  Then we can construct a point $c'=(E,P',Q')$ on $X_1(m',m'n')$ of degree $d$ over $\Q$ with $P\in\langle P'\rangle,\, Q\in\langle Q'\rangle$, and $\langle P,Q\rangle\subsetneq \langle P',Q'\rangle$, and we have maps
\[
X_1(B,B) \overset{\pi_1}{\longrightarrow} X_1(m',m'n')\overset{\pi_2}{\longrightarrow}X_1(m,mn)\overset{f}{\longrightarrow} \P^1,
\]
in which $\pi=\pi_2\circ\pi_1$, with $\deg\pi_1 < \deg\pi$, and $c=\pi_2(c')$.
If we now consider $b\in\pi_1^{-1}(c')$, then $b\in\varphi^{-1}(a)$ has degree $d\deg \pi_1<d\deg\pi$ over $\Q$, a contradiction.
\end{proof}

We now outline the strategy of the proof.  Let $T$ be the set of pairs $(m,mn)$ identifying torsion subgroups that we wish to prove is equal to $\Phi^\infty(d)$.
To prove $\Phi^\infty(d)=T$ we proceed as follows.

\begin{enumerate}
\item Prove that $\Q(\zeta_m)(X_1(m,mn))$ contains a function of degree $d/\phi(m)$ for all $(m,mn)\in T$.
\item Compute the set $T_1\coloneqq \{(m,mn): \phi(m)|d\text{ and }B(m,mn)\le 2d\}$ where $B(m,mn)$ is the lower bound on $\gamma(X_1(m,mn))$ given by \Cref{thm:Abramovich} (we have $\Phi^\infty(d)\subseteq T_1$).

\item Verify that $\rk(J_1(m,mn)(\Q(\zeta_m)))=0$ for all $(m,mn)\in T_1$ via \Cref{thm:rank0}.
\item Compute the set $T_2\coloneqq \{(m,mn): \phi(m)|d\text{ and }B(m,mn)\le d\}\subseteq T_1$ (now $\Phi^\infty(d)\subseteq T_2$).
\item For $(m,mn)\in T_2-T$ prove that $\Q(\zeta_m)(X_1(m,mn))$ has no functions of degree $d/\phi(m)$.
\end{enumerate}
In step (2) the restriction on $m$ follows from the Weil pairing, and the restriction on $n$ is from \Cref{thm:Frey}; the tighter restriction on $n$ in step (4) is \Cref{lem:rank0}.
If the verification in step~(3) succeeds, then \Cref{cor:rank0} and \Cref{lem:suff} together imply that each pair $(m,mn)\in T_2$ lies in $\Phi^\infty(d)$ if and only if $\Q(\zeta_m)(X_1(m,mn))$ contains a function of degree $d/\phi(m)$. 

\begin{remark}\label{remark:higher_d}
We have completed steps (1)-(4) for $d=5,6,7,8$.
This strategy cannot be applied with $d = 9$ because (3) fails; as proved in \cite{DvH13}, we have $\gamma(X_1(37))=18$ and $\rk(J_1(37)(\Q))\ne 0$.
\end{remark}

We now prove \Cref{thm:main}, beginning with the case $d=5$, following the strategy above.

\begin{proposition}\label{prop:d5}
$\Phi^\infty(5) = \{(1,n):1\le n\le 25, n\ne 23\}\cup \{(2,2n):1\le n\le 8\}$.
\end{proposition}
\begin{proof}
The existence of the Weil pairing implies that $(m,mn)\in\Phi^\infty(5)$ only if $\phi(m)$ divides $5$; we thus have $m\le 2$.
The elements $(1,n)\in \Phi^\infty(5)$ are determined in \cite[Thm.\,3]{DvH13}, so we only need to consider $m=2$.

For $1\le n\le 6$ the genus of $X_1(2,2n)$ is either 0 or 1 and $X_1(2,2n)(\Q)\ne\emptyset$; it follows that $\Q(X_1(2,2n))$ contains functions of every degree $d\ge 2$, including $d=5$.

For $n=7$ we used the method of \S\ref{sec:specialmodels} and a modified version of the algorithm in \cite{Sutherland12} to construct the model
\[
X_1(2,14): v^3 - (u^3 + u^2 + u - 1)v^2 - (u^5 + 3u^4 + 3u^3 + u^2 + u)v + u^5 + u^4 = 0,
\]
with maps
\[
q=\frac{v + 1}{v-2u + 1},\qquad t=\frac{(v+1)(2u-v+1)(2u(u+1)+v+1)}{v^3 + (2u^2 + 1)v^2 - (2u^3 - 2u^2 - 2u - 1)v + u^4 + (u+1)^4}
\]
that give points $(E_2(q,t),P_2(q,2),Q_2(q,t))$ on $X_1(2,14)$ as defined in \S\ref{sec:specialmodels}.
We may then take~$u$ as a function of degree 5 in $\Q(X_1(2,14))$.

For $n=8$ we similarly constructed
\[
X_1(2,16): v^4 + (u^3 - 2u)v^3 - (2u^4 + 2)v^2 + (u^5 + u^3 + 2u)v + 1 = 0,
\]
which also has $u$ as a function of degree 5 in $\Q(X_1(2,16))$ (we omit the maps $q(u,v)$ and $t(u,v)$ for reasons of space).
This completes step (1) of our proof strategy.

Proceeding to steps (2)--(4), from \Cref{thm:Abramovich} we find that $\gamma(X_1(2,2n))>10$ for $n>18$, and $\gamma(X_1(2,2n))>5$ for $n>13$.
By \Cref{thm:rank0} we have $\rk(J_1(2,2n)(\Q))=0$ for $n\le 18$, thus by \Cref{lem:rank0} 
it suffices to prove that $\Q(X_1(2,2n))$ contains no functions of degree 5 for $9\le n\le 13$.

For step (5) we begin by proving that $\gamma(X_1(2,2n)_{\F_3})>5$ for $n=10,11,13$, and that $\gamma(X_1(2,24)_{\F_5})>5$ for $n=12$, using methods similar to those in \cite{DvH13}.
This involves exhaustively searching the Reimann--Roch spaces of a suitable set divisors for functions of degree $d\le 5$; the Magma code we used to perform these computations can be found in \cite{mdmagma}.

For $n=9$ we actually have $\gamma(X_1(2,18))=4$, so in this case we need to show that $\Q(X_1(2,18))$ contains no functions of degree exactly equal to 5; this is addressed by \Cref{prop:d5m2n9} below.
\end{proof}

\begin{proposition}\label{prop:d5m2n9}
$\Q(X_1(2,18))$ does not contain a function of degree $5$
\end{proposition}
\begin{proof}
We proceed by verifying conditions 1--5 of \cite[Prop,\,7]{DvH13} for $p=d=5$.
For $k=\Q,\F_5$, let $W_d^r(k)$ denote the closed subscheme of $\Pic^d(X_1(2,18)_k(k))$ corresponding to line bundles of degree $d$ whose global sections form a $k$-vector space of dimension strictly greater than $r$.
\begin{itemize}
\item[1.] The map $J_1(2,18)(\Q)\to J_1(2,18)_{\F_5}(\F_5)$ is injective because $J_1(2,18)(\Q)$ is finite.
\item[2.] Using Magma, a brute force search of the Riemann--Roch space of all effective divisors of degree $5$ on $X_1(2,18)_{\F_5}$ finds that $\F_5(X_1(2,18)_{\F_5})$ contains no functions of degree~5.
\item[3.] A similar brute force computation shows that $W_5^2(\F_5)=\emptyset$.
\item[4.] A brute force computations finds that $5-\gamma(X_1(2,18)_{\F_5})=5-4=1$ and $\#W_4^1(\F_5)=3$.
The surjectivity of the map $W_4^1(\Q) \to W_4^1(\F_5)$ is verified in \cite{mdmagma} by finding three modular units of degree $4$ on $X_1(2,18)$ with linearly independent pole divisors.
\item[5.] The reduction map $X_1(2,18)(\Q)\to X_1(2,18)_{\F_5}(\F_5)$ is surjective because the 9 elements of $X_1(2,18)_{\F_5}(\F_5)$ are precisely the reductions of the 9 cusps in $X_1(2,18)(\Q)$.
\end{itemize}
It follows from \cite[Prop.\,7]{DvH13} that $\Q(X_1(2,18))$ contains no functions of degree 5.
\end{proof}

We now address $d=6$ using the same proof strategy.

\begin{proposition}\label{prop:d6}
We have
\begin{align*}
\Phi^\infty(6)= &\{(1,n):1\le n\le 30, n\ne 23,25,29\}\ \cup\ \{(2,2n):1\le n\le 10\}\\
&\cup\ \{(3,3n):1\le n\le 4\}\ \cup\ \{(4,4),(4,8),(6,6)\}.
\end{align*}
\end{proposition}
\begin{proof}
We have $\phi(m)$ dividing $6$ for $m=1,2,3,4,6,7,9,14,18$.
The elements $(1,n)\in \Phi^\infty(6)$ are determined in \cite[Thm.\,3]{DvH13}, and we can immediately rule out $m=7,9,14,18$, since for these $m$ we have $\phi(m)=6$ and $d/\phi(m)=1$, but $g(X_1(m,mn))>1$ for all $m>6$ and $n\ge 1$.  Faltings' theorem implies that $X_1(m,mn)(\Q(\zeta_m))$ is finite for $m>6$.
This leaves only $m=2,3,4,6$.

We have $\Phi^\infty(2),\Phi^\infty(3)\subseteq \Phi^\infty(6)$ (given a degree 2 or 3 point $(E,P,Q)$ on $X_1(m,mn)$ we can always base change $E$ to a number field of degree 6).
It thus suffices to show that $\Q(\zeta_m)(X_1(m,mn))$ contains a function of degree 6 for the pairs $(2,16),(2,18),(2,20)$, and a function of degree 3 for the pairs $(3,9),(3,12),(4,8),(6,6)$.

The map $\pi\colon X_1(2,2n)\to X_1(2n)$ given by $(E,P,Q)\mapsto (E,Q)$ has degree 2.
For $n=8,9,10$ we have a function $f$ of degree 3 in $\Q(X_1(2n))$, because $(1,2n)\in \Phi^\infty(3)$ and $\rk(J_1(2n)(\Q))=0$ (by \Cref{thm:rank0}), and $\pi\circ f$ is then a function of degree 6 in $\Q(X_1(2,2n))$.

The curves $X_1(3,9)$, $X_1(4,8)$, $X_1(6,6)$ all have genus 1 and a $\Q(\zeta_m)$-rational point (take a rational cusp), hence they are isomorphic to elliptic curves and admit functions of every degree $d\ge 2$.  The map $\colon X_1(3,12)\to X_1(12)_{\Q(\zeta_3)}$ given by $(E,P,Q)\mapsto (E,Q)$ has degree 3 and $X_1(12)_{\Q(\zeta_3)}$ has genus 0 and a $\Q(\zeta_3)$-rational point, hence it is isomorphic to $\P^1_{\Q(\zeta_m)}$; it follows that $\Q(\zeta_3)(X_1(3,12))$ contains a function of degree~3.

This completes step (1) of our proof strategy.  We now proceed to steps (2)--(5) for each $m=2,3,4,6$ in turn.

We begin with $m=2$.
From \Cref{thm:Abramovich} we find that $\gamma(X_1(2,2n))>12$ for $n>21$, and $\gamma(X_1(2,2n))>6$ for $n>15$.
By \Cref{thm:rank0} we have $\rk(J_1(2,2n)(\Q))=0$ for $n\le 21$, thus by \Cref{lem:rank0} 
it suffices to prove that $\Q(X_1(2,2n))$ contains no functions of degree 6 for $11\le n\le 15$.
For $11\le n\le 14$ we proceed as in \Cref{prop:d5} to prove $\gamma(X_1(2,2n)_{\F_p})>6$ using $p=3$ for $n=11,13,14$ and $p=5$ for $n=12$, which was done with a computation in Magma. The code written for these and the subsequent computations can be found in \cite{mdmagma}.
For $n=15$ we instead apply \Cref{lem:d6m2n15} and \Cref{prop:d6m2n15}.

We next consider $m=3$.
We have $\gamma(X_1(3,3n))>6$ for $n>8$ and $\gamma(X_1(3,3n))>3$ for $n>5$, and $\rk(J_1(3,3n)(\Q(\zeta_3))=0$ for $n\le 8$.
It suffices to show that $\Q(\zeta_3)(X_1(3,15))$ contains no functions of degree 3, which again follows from a Magma computation.

The cases $m=4,6$ are similar.
For $m=4$ we have $\gamma(X_1(4,4n))>6$ for $n>6$ and $\gamma(X_1(4,4n))>3$ for $n>3$, and $\rk(J_1(4,4n)(\Q(\zeta_4))=0$ for $n\le 6$; a computation shows that $\gamma(X_1(4,12)_{\F_5})>3$.
For $m=6$ we have $\gamma(X_1(6,6n))>6$ for $n>2$ and $\rk(J_1(6,6n)(\Q(\zeta_6))=0$ for $n\le 2$; a Magma computation shows that $\gamma(X_1(6,12)_{\F_5})>3$.
\end{proof}

\begin{lemma}$X_1(2,30)$ has no non-cuspidal points of degree less than $6$. \label{lem:d6m2n15}
\end{lemma}
\begin{proof}
$X_1(2,30)$ has a degree two map to $X_1(30)$ given by $(E,P,Q)\mapsto (E,Q)$, so we start by considering the points of degree less than 6 on $X_1(30)$.

Theorem~3 of \cite{DvH13} states that $X_1(30)$ has only finitely many points of degree less than $6$, and these points are explicitly determined in \cite{DvHnp}.
The non-cuspidal points on $X_1(30)$ of degree less than 6 all have degree 5 and arise from two Galois conjugacy classes of elliptic curves that we now describe.

Let $x_1, x_2 \in \overline \Q$ be zeros of $x^5+x^4-3x^3+3x+1$ and $x^5+x^4-7x^3+x^2+12x+3$, respectively, and define 
\[
y_1 := 2x_1^4 + x_1^3 - 6x_1^2 + 4x_1 +4, \quad y_2 :=  \frac {3x_2^4+7x_2^3+6x_2^2+11x_2-73} {53}.
\]
Let $E_{x,y}$ denote the curve $E(b,c)$ in Tate normal form with $b=rs(r-1)$ and $c=s(r-1)$, where $r=(x^2y-xy+y-1)/(x^2y-x)$ and $s:=(xy-y+1)/(xy)$.

The point $P_0 := (0,0)$ is a point of order 30 on both $E^{}_{x_1,y_1}$ and $E^{}_{x_2,y_2}$.
Moreover, every non-cuspidal point $(E,P) \in X_1(\overline \Q)$ of degree less than 6 over $\Q$ can be obtained as either  $(E_{\sigma(x_1),\sigma(y_1)},dP_0)$ or $(E_{\sigma(x_2),\sigma(y_2)},dP_0)$ for some $\sigma \in \Gal(\Qbar/\Q)$ and $d \in (\Z/30\Z)^\times$; thus every such point has degree 5.

The 2-division polynomials of $E^{}_{x_1,y_1}$ and $E^{}_{x_2,y_2}$ each have just one root in $\Q(x_1)$ and $\Q(x_2)$, respectively, corresponding to $x(15P_0)$; it follows that $E^{}_{x_1,y_1}(\Q(x_1))$ and $E^{}_{x_2,y_2}(\Q(x_2))$ contain no other points of order 2. Thus every non-cuspidal point in $X_1(2,30)(\overline \Q)$ that maps to a point of degree less than 6 in $X_1(30)(\overline \Q)$ has degree at least $2\cdot 5 = 10$. The lemma follows.
\end{proof}

\begin{proposition}\label{prop:d6m2n15}
$X_1(2,30)$ has only finitely many points of degree $6$.
\end{proposition}
The proof below is based on ideas presented in \cite[\S 4]{DKM16}. 
\begin{proof}
We have $d(X_1(2,30))=\gamma(X_1(2,30))$ by \Cref{lem:rank0}, since $\rk(J_1(2,30)(\Q))=0$, by \Cref{thm:rank0}, and $d(X_1(2,30))=\gamma(X_1(2,30))\ge 6$, by \Cref{lem:d6m2n15}.
It thus suffices to prove there are no functions of degree $6$ in $\Q(X_1(2,30))$.

We first show that if $\Q(X_1(2,30))$ contains a function of degree 6 then it contains a function whose pole and zero divisors consist entirely of cusps. Let $f \in \Q(X_1(2,30))$ be a function of degree $6$ and let $c \in X_1(2,30)(\Q)$ be any of its 12 rational cusps.
We may assume $f$ has a pole at $c$ by replacing $f$ with $1/(f -f(c))$ if necessary.
The pole divisor of $f$ can then be written as $c + D$, where $D$ is a divisor of degree $5$.
Now let $c'\ne c$ be a rational cusp not in the support of~$D$, and define $f' := f -f(c')$.
The function $f'$ has the same poles as $f$, and its zero divisor can be written as $c'+D'$ with $D'$ a divisor of degree $5$. The divisors $D$ and $D'$ must consist entirely of cusps, by \Cref{lem:d6m2n15}, and the claim follows.

An enumeration of all $f \in \Q(X_1(2,30))$ of degree at most 6 with $\Div(f)$ supported on cusps computed as in \cite[Footnote 7]{DvH13} finds none with degree exactly 6; the proposition follows.
\end{proof}

\section*{Acknowledgements}
The authors thank the referee for several helpful remarks on an earlier draft of this paper.


\begin{thebibliography}{9}

\bibitem{Abramovich96}
Dan Abramovich, \href{http://imrn.oxfordjournals.org/content/1996/20/1005}{\textit{A linear lower bound on the gonality of modular curves}}, Internat. Math. Res. Notices \textbf{20} (1996), 1005-1011.

\bibitem{AtkinLi78}
A.O.L. Atkin and Wen-Ch'ing Winnie Li, \href{http://gdz.sub.uni-goettingen.de/dms/load/img/?PID=GDZPPN002094568}{\textit{Twists of newforms and pseudo-eigenvalues of $W$-operators}}, Invent. Math. \textbf{48} (1978), 221--243.

\bibitem{magma}
Wieb Bosma, John Cannon, and Catherine Playoust, \href{http://www.sciencedirect.com/science/article/pii/S074771719690125X}{\textit{The Magma algebra system I: The user language}},  J. Symbolic Comput. \textbf{24} (1997), 235--265.

\bibitem{CES03}
Brian Conrad, Bas Edixhoven, William Stein,
\href{https://www.math.uni-bielefeld.de/documenta/vol-08/12.pdf}{\textit{$J(p)$ has connected fibers}}, Doc. Math. \textbf{8} (2003), 331--408.

\bibitem{DR73}
Pierre Deligne and Michael Rapoport, \href{http://link.springer.com/chapter/10.1007/978-3-540-37855-6_4}{\textit{Les sch\'emas de modules de courbes elliptiques}}, in \textit{Modular functions of one variable, {II}}, Lecture Notes in Math. \textbf{349} (1973), 143--316.

\bibitem{DvH13}
Maarten Derickx and Mark van Hoeij, \href{http://www.sciencedirect.com/science/article/pii/S0021869314003585}{\textit{Gonality of the modular curve $X_1(N)$}}, J. Algebra \textbf{417} (2014), 52--71.

\bibitem{DvHnp}
Maarten Derickx and Mark van Hoeij, data related to \cite{DvH13}, available at \url{www.math.fsu.edu/~hoeij/files/X1N}, 2013.

\bibitem{DKM16}
Maarten Derickx, Sheldon Kamienny, and Barry Mazur, \textit{Rational families of $17$-torsion points of elliptic curves over number fields}, Contemporary Mathematics, AMS, to appear.

\bibitem{mdmagma}
Maarten Derickx and Andrew V. Sutherland, GitHub repository related to \textit{Torsion subgroups of elliptic curves over quintic and sextic number fields}, available at \url{https://github.com/koffie/mdmagma/}, 2016.

\bibitem{models}
Maarten Derickx and Andrew V. Sutherland, \href{http://math.mit.edu/~drew/X1mn.html}{\textit{Explicit models of $X_1(m,mn)$}}, available at \url{http://math.mit.edu/~drew/X1mn.html}, 2016.

\bibitem{Frey94}
Gerhard Frey, \href{http://link.springer.com/article/10.1007/BF02758637}{\textit{Curves with infinitely many points of fixed degree}}, Israel J. Math. \textbf{85} (1994), 79--83.

\bibitem{Faltings94}
Gerd Faltings,\textit{The general case of S. Lang's conjecture}, in \textit{Barsotti Symposium in Algebraic Geometry (Abano Terme, 1991)}, Perspectives in Mathematics \textbf{15} (1994), Academic Press, 175--182. 

\bibitem{Jain10}
Sonal Jain, \href{http://dx.doi.org/10.1112/S1461157009000151}{\textit{Points of low height on elliptic surfaces with torsion}}, LMS J. Comput. Math. \textbf{13} (2010), 370--387.

\bibitem{JKL11a}
Daeyeol Jeon, Chang Heon Kim, and Yoonjin Lee, \href{http://www.ams.org/journals/mcom/2011-80-273/S0025-5718-10-02369-0/S0025-5718-10-02369-0.pdf}{\textit{Families of elliptic curves over cubic number fields with prescribed torsion subgroups}}, Math. Comp. \textbf{80} (2011) 579--591.


\bibitem{JKP06}
Daeyeol Jeon, Chang Heon Kim, and Euisung Park, \href{http://dx.doi.org/10.1112/S0024610706022940}{\textit{On the torsion of elliptic curves over quartic number fields}}, J. London Math. Soc. \textbf{74} (2006) 1--12.

\bibitem{JKS04}
Daeyeol Jeon, Chang Heon Kim, and Andreas Schweizer, \href{http://dx.doi.org/10.4064/aa113-3-6}{\textit{On the torsion of elliptic curves over cubic number fields}}, Acta Arith. \textbf{113} (2004) 291--301.

\bibitem{Kamienny92}
Sheldon Kamienny, \href{http://link.springer.com/article/10.1007/BF01232025}{\textit{Torsion points on elliptic curves and $q$-coefficients of modular forms}}, Invent. Math. \textbf{109} (1992) 221-229.

\bibitem{Kato04}
Kazuya Kato, \href{http://smf4.emath.fr/en/Publications/Asterisque/2004/295/html/smf_ast_295_117-290.html}{\textit{$p$-adic Hodge theory and values of zeta functions of modular forms}}, Ast\'erisque. \textbf{295} (2004), 117--290.

\bibitem{KM88}
M. A. Kenku and Fumiyuki Momose, \href{https://projecteuclid.org/euclid.nmj/1118780896}{\textit{Torsion points on elliptic curves defined over quadratic fields}}, Nagoya Math. J. \textbf{109} (1988), 125--149.

\bibitem{KS03}
Henry H. Kim, \href{http://dx.doi.org/10.1090/S0894-0347-02-00410-1}{\textit{Functoriality for the exterior square of {${\rm GL}_4$} and the symmetric fourth of {${\rm GL}_2$}}}, with appendix~1 by D. Ramakrishnan and appendix~2 by H. Kim and P. Sarnak,  J. Amer. Math. Soc. \textbf{16} (2003), 139--183.

\bibitem{KL89}
V. A. Kolyvagin and D. Yu. Legach\"ev, \href{http://www.ams.org/mathscinet-getitem?mr=1036843}{\textit{Finiteness of the Shafarevich-Tate group and the group of rational points for some modular abelian varieties}}, Algebra i Analiz \textbf{1} (1989), 171--196; translation in 
Leningrad Math. J. \textbf{1} (1990), 1229--1253.

\bibitem{Kubert76}
Daniel S. Kubert, \href{http://plms.oxfordjournals.org/content/s3-33/2/193.full.pdf}{\textit{Universal bounds on the torsion of elliptic curves}}, Proc. London Math. Soc. \textbf{33} (1976), 193--237.

\bibitem{KS13}
Abhinav Kumar and Tetsuji Shioda, \href{http://msp.org/ant/2013/7-7/ant-v7-n7-p03-s.pdf}{\textit{Multiplicative excellent families of elliptic surfaces of type $E_7$ or $E_8$}}, Algebra and Number Theory \textbf{7} (2013), 1613--1641.

\bibitem{Mazur77a}
Barry Mazur, \href{http://www.numdam.org/item?id=PMIHES_1977__47__33_0}{\textit{Modular curves and the Eisenstein ideal}}, Inst. Hautes \'Etudes Sci. Publ. Math. \textbf{47} (1977) 33--186.


\bibitem{Merel96}
Lo\"ic Merel, \href{http://link.springer.com/article/10.1007/s002220050059}{\textit{Bornes pour la torsion des courbes elliptiques sur les corps de nombres}}, Invent. Math. \textbf{124} (1996), 437--449.


\bibitem{Reichert86}
Markus A. Reichert,
\href{http://www.ams.org/journals/mcom/1986-46-174/S0025-5718-1986-0829635-X/}{\textit{Explicit determination of nontrivial torsion structures of elliptic curves over quadratic number fields}}, Math. Comp. \textbf{46} (1986), 637--658.


\bibitem{Silverman09}
Joseph H. Silverman, \href{http://link.springer.com/book/10.1007/978-0-387-09494-6}{\textit{The arithmetic of elliptic curves}}, Springer-Verlag, 2nd Edition, New York, 2009.

\bibitem{Sutherland12}
Andrew V. Sutherland, \href{http://www.ams.org/journals/mcom/2012-81-278/S0025-5718-2011-02538-X/}{\textit{Constructing elliptic curves over finite fields with prescribed torsion}}, Math. Comp. \textbf{81} (2012), 1131--1147.

\bibitem{Weil29}
Andr\'e Weil, \href{http://link.springer.com/article/10.1007/BF02592688}{\textit{L'arithm\'etique sur les courbes alg\'ebriques}}, Acta Math. \textbf{52} (1929) 281–315.

\end{thebibliography}
\end{document}